\documentclass[a4paper,10pt]{article}
\usepackage[latin1]{inputenc}
\usepackage{amsmath, amscd, amsthm, amssymb, amsfonts}
\usepackage{xy}
\newtheorem{remark}{Remark}
\newtheorem{lemma}{Lemma}[section]
\newtheorem{theorem}[lemma]{Theorem}
\newtheorem{proposition}[lemma]{Proposition}

\theoremstyle{definition}
\newtheorem{example}{Example}[section]
\DeclareMathOperator{\ini}{in}

\DeclareMathOperator{\Hs}{Hs}
\DeclareMathOperator{\Hf}{Hf}

\DeclareMathOperator{\post}{post}

\DeclareMathOperator{\nz}{nz}

\DeclareMathOperator{\V}{V}
\DeclareMathOperator{\nf}{Nf}
\DeclareMathOperator{\NZ}{NZ}

\DeclareMathOperator{\diag}{diag}
\DeclareMathOperator{\rank}{rank}
\DeclareMathOperator{\initials}{Initials}
\title{Multiplication matrices and ideals of projective dimension zero}
\date{}
\author{Samuel Lundqvist}

\begin{document}

\maketitle

\begin{abstract}
We introduce the concept of multiplication matrices for ideals of projective dimension zero. We discuss various applications and in particular, we give a new algorithm to compute the variety of an ideal of projective dimension zero. 
\end{abstract}

\section{Introduction}
Eigenvalue methods to compute the variety of an affine zero-dimensional ideal has become an active area of research \cite{corless, Corlessetal, Mollerstetter, Mollertenberg}. Recall that an ideal is defined to be of dimension zero when the corresponding variety is finite.
The eigenvalue methods use both algebraic and numerical algorithms and the key is a nice one-to-one correspondence between the points on the variety and the eigenvectors to so called multiplication matrices.

The notion of zero-dimensionality has a meaning also in projective space. Over an infinite field, an ideal of projective dimension zero is an ideal whose variety consists of a finite number of projective points. Below we will give a general definition, valid also for finite fields.
We show that it is possible to define multiplication matrices with respect to ideals of projective dimension zero. Our main result is that the one-to-one correspondence mentioned above also holds in the projective setting, giving a new method to compute the variety of an ideal of projective dimension zero. 

%Based on the multiplication matrices, we decompose quotient rings 
%$\Bbbk[x_0, \ldots, x_n]/I$, where $I$ is an ideal of projective dimension zero into a triplet 
%$(\tilde{R}, , )
In order to define the projective multiplication matrices, we need to choose appropriate vector space bases for the graded pieces of the quotient ring $S/I$, where $S$ is a polynomial ring over some field and $I$ is an ideal of projective dimension zero. 

The usual choice of a vector space basis is the set of residues to the complement of the initial ideal of I (with respect to some monomial order). Our choice of bases differs from the usual ones --- in general we consider non-monomial $\Bbbk$-bases.
This choice of bases happen to give a fast normal form algorithm for high degree elements, which outperforms the usual reduction method based on Gr\"obner bases. 

Moreover, we discuss how our approach can be used to compute vanishing ideals of projective points. We give an alternative version of the graded Buchberger-M\"oller algorithm and give a fast algorithm for computing projective separators. 
%Finally, we give a formula for computing the radical of an ideal of projective dimension zero.

%Parts of our work might be seen as a criticism of the use of Gr\"obner theory in computational %algbera, and indeed,  our philosophy about Gr\"obner bases is that they should be computed when %needed. 

% In particular, we argue that computations over ideals defined by functionals, for instance vanishing %ideals of points, are better made without a Gr\"obner basis.  

\section{Notation and preliminaries}
Throughout the paper, let $\Bbbk$ be a field and let $S = \Bbbk[x_0, \ldots, x_n]$ denote the polynomial ring in $n+1$ variables. Recall that the Hilbert series of a graded ring $R = R_0 \oplus R_1 \oplus R_2 \oplus \cdots$ is the power series $\Hs(R,t) = \dim_{\Bbbk}(R_0) + \dim_{\Bbbk}(R_{1})t + \dim_{\Bbbk}(R_{2})t^2 + \cdots.$ An ideal $I$ is of projective dimension zero exactly when $R=S/I$ is graded and satisfies $\dim_{\Bbbk} (R_i) = m$ for some $m>0$ and for all $i$ sufficiently large.
 The least $i$ such that
$\dim_{\Bbbk} (R_{i}) = \dim_{\Bbbk} (R_{i+1}) = \cdots$ is called the \emph{postulation number} and is denoted by $\post(R)$. We say that $R$ postulates in degree $\post(R)$.
 
When $I$ is an ideal of projective dimension zero and $R=S/I$, we say that $R$ is a ring of projective dimension zero. For an element $a$ of $S$, we write $[a]$ to denote the equivalence class in $R$ containing $a$.
 
 By $\overline{V(I)}$, we denote the
variety of $I$ with respect to the algebraic closure $\overline{\Bbbk}$ of $\Bbbk$, so $|V(I)| \leq |\overline{V(I)}|$. 
The number of projective points in $\overline{V(I)}$ counting multiplicity equals $\dim_{\Bbbk}(R_{\post(R)})$. 
% Recall that the variety of an ideal pf projective dimension zero can actually be empty, consider for instance
% $I = x_1^2+x_2^2 \subset \mathfrak{Q}[x_1, x_2]$. Then $V(I) = \theemptyset$, while over $\mathfrak{C}$, $V(I) = \{(1:i), (1:-i)\}$.
% 
% We will however, in many cases, we will assmue ork only with rational points in this paper. A rational point This means that although $\Bbbk$ is not algebraically closed, we will assume that
% each point in the variety of $I$ in the algebraic closure of $\Bbbk$ takes its coordinate values in $\Bbbk$. 

Let $I = Q_1 \cap \cdots \cap Q_{s}$ be a minimal primary decomposition of $I$. Let $P_i = \sqrt{Q_i}$. When no $P_i$ equals the unique graded maximal
 ideal $\mathfrak{m} = (x_0, \ldots, x_n)$ of $S$, we say that $I$ is unmixed. When $I$ is unmixed, the Hilbert series of $R$ is strictly increasing until
it reaches degree $\post(R)$. 

When $I$ is mixed, we will write the primary decomposition as 
$I = Q_1 \cap \cdots \cap Q_{s} \cap Q_{s+1}$ and assume that $\mathfrak{m} = \sqrt{Q_{s+1}}$. The Hilbert series of $R$ when $I$ is mixed does not behave nice in general (it can have an arbitrary number of valleys for instance). 

We need a notation to drop the primary component $Q_{s+1}$ in the case when $I$ is mixed, so we define $I^u = Q_1 \cap \cdots \cap Q_s$. This means that when $I$ is unmixed, we have the identity $I = I^u$. We extend this definition so that $R^u = S/I^u$. 

\begin{example} \label{ex:not}
Let $I = x_1 x_2^3 - x_2^4, x_1^3 x_2^2 - x_2^5$. Then 
$I = (x_1-x_2) \cap (x_2^2) \cap (x_1^3,x_2^3)$
Thus, $Q_1 = (x_1-x_2)$, $Q_2 = (x_2^2)$ and $Q_3 = (x_1^3,x_2^3)$. We have
$P_1 = \sqrt{Q_1} = Q_1$, so $V(P_1) = \{(1:1)\}$ and 
$P_2 = \sqrt{Q_2} = (x_2)$ so $V(P_2) = \{(1:0)\}$. Finally 
$\mathfrak{m} = \sqrt{Q_3}$. Thus, $I^u = (x_1-x_2) \cap (x_2^2) = (x_1x_2^2-x_2^3).$
We have $\Hs(R^u,t) = 1+2t+3t^2+3t^3+\cdots$, while
$  \Hs(R,t) = 1+2t+3t^2+4t^3+4t^4+3t^5+3t^6 +\cdots$, so $\post(R^u) = 2$ and 
$\post(R) = 5$ (It does not in general hold though that $\post(R^u) \leq \post(R)$.) The point $(1:1)$ has multiplicity one, while the point $(1:0)$ has multiplicity two (it follows from the fact that $(x_2^2)$ has a chain of primary ideals of length two).
\end{example}
We have a one-to-one correspondence of prime ideals $P_i$ generated in degree one and points on $V(I)$. 
If $\overline{V(I)} = V(I)$, it follows that $|V(I)| = s$ and that $P_1, \ldots, P_s$ are all generated in degree one. 

The concept of non-zero divisors is of particular importance in this paper. Recall that $l$ is a non-zero divisor on the $S$-module $M$ if
$lm=0$ implies that $l \neq 0$. In Example \ref{ex:not}, $[x_1]$ is a non-zero divisor on $R^u$, while $[x_2]$ is not.
The existence of non-zero divisors is connected to Cohen-Macaulayness and also to the primary decomposition in the following sense for  
a ring $R = S/I$ of projective dimension zero.
$$ I \text{ is unmixed} \Leftrightarrow R \text{ is Cohen-Macaulay} \Leftrightarrow R \text{ contains  a non-zero divisor} $$

Although $R$ lacks non-zero divisors when $I$ is mixed, 
we will show that there exists a minimal degree $d$ so that $R^u_d \oplus R^u_{d+1} \oplus \cdots$ and  $R_{d} \oplus R_{d+1} \oplus \cdots$ are isomorphic as $S$-modules. 
This degree $d$ will equal $\max(\post(R), \post(R^u))$. To simplify notation, we will denote $\max(\post(R), \post(R^u))$ by $\nz(R)$. When the context is clear, we will omit $R$ and only write $\nz$. It follows that every non-zero divisor on $R^u$ is 
also a non-zero divisor on $R_{\nz} \oplus R_{\nz+1} \oplus \cdots$. In Example \ref{ex:not}, $\nz(R) = 5$ and $[x_1]$ is a non-zero divisor on $R_5 \oplus R_6 \oplus \cdots$.

Suppose that $V$ and $W$ are two $\Bbbk$-spaces of dimension $m$ and $m'$ respectively. Let $e_1, \ldots, e_m$ and $f_1, \ldots, f_{m'}$ be $\Bbbk$-bases of $V$ and $W$, respectively. Let $\phi$ be a $\Bbbk$-linear map from $V$ to $W$ and let $A_{\phi}$ be the $m \times m'$ matrix, whose $i$'th row is the coefficient vector $(c_1, \ldots, c_{m'})$ and where $\phi(e_i) = c_1 f_1 + \cdots + c_{m'} f_{m'}.$ Notice that $A_{\phi}$ is the transpose of the standard matrix representation of $\phi$. 

When the map $\phi$ is defined on an a finite dimensional algebra by $v \mapsto fv$ for an element $f$ in the algebra, then the matrix $A_{\phi}$ is called the \emph{multiplication matrix} with respect to $f$.  We also denote $A_{\phi}$ by $A_{f}$

\begin{example}
Let $I = (x_1-1,x_2^2-x_2) \subset \mathbf{Q}[x_1,x_2]$. The ring $\mathbf{Q}[x_1,x_2]/I$ is of affine dimension zero and a vector space basis for the quotient ring can be
chosen as $[1],[x_2]$. In this algebra we have $[x_1] [1] = [1], [x_1] [x_2] = [x_2], [x_2][1] = [x_2]$ and $[x_2] [x_2] = [x_2]$, so the multiplication matrices with respect to  $[x_1]$ and $[x_2]$ equals
$$
A_{x_1} = 
\begin{pmatrix}
1 & 0 \\
0 & 1
\end{pmatrix} \text{ and }
A_{x_2} = 
\begin{pmatrix}
0 & 1 \\
0 & 1
\end{pmatrix}.
$$
It is customary to write $A_{x_i}$ instead of $A_{[x_i]}$ and we will do so in the sequel. 
\end{example}
When it comes to rings of projective dimension zero, we have a problem since these rings are infinite dimensional. In order to overcome this problem, 
we use the graded property among these rings. So let $\phi$ be defined from $R_d$ to $R_{d+|f|}$ 
by multiplication by a form $f$ with respect to the bases $e_1, \ldots, e_m$ and $f_1, \ldots, f_{m'}$ for $R_d$ and $R_{d+|f|}$ respectively. Then $A_{f}$ (or $A_{\phi}$) is
the \emph{projective multiplication matrix} in degree $d$ with respect to $f$. We will later show that it is possible to choose bases such that the projective multiplication matrices agrees for all degrees
greater than or equal to $\nz(R)$.
 
\begin{example}
 For the ring $R^u$ from Example \ref{ex:not}, we can choose $[x_1^2] ,[x_1x_2],[x_2^2]$ as a $\Bbbk$-basis in degree two and 
$[x_1^3], [x_1^2x_2], [x_2^3]$ as a $\Bbbk$-basis in degree three. If $\phi$ denotes the map from
$R^u_2$ to $R^u_3$ induced by multiplication by $[x_2]$, then the projective multiplication matrix of degree two with respect to $[x_2]$ equals

$$
\begin{pmatrix}
0 $ 1 $ 0\\
0 $ 0 $ 1\\
0 $ 0 $ 1
\end{pmatrix}
$$
since 
$[x_2]  [x_1^2] = 0 [x_1^3] + 1 [x_1^2x_2] + 0 [x_2^3]$,  
$[x_2] [x_1x_2] = 0 [x_1^3] + 0 [x_1^2x_2] + 1 [x_2^3]$
and
$[x_2][x_2^2] = 0 [x_1^3] + 0 [x_1^2x_2] + 1 [x_2^3]$.

\end{example}

% We will also show n $\Bbbk$ contains enough elements and $I$ is unmixed, we will show that $R$ will be Cohen-Macaulay. A property that is equivalent to being Cohen-Macaulay is the existence
% of a non-zero divisor.  
% there exists an element $l \in S$ such that $[l]$ is a non-zero divisor on $R$. 
% 
% When $I$ is mixed and $\Bbbk$ contains enough elements, we will show that there will always exist an element
% $l \in S_1$ and a degree $d$ such that [l] is a non-zero divisor on 
% the $R$-module $R_{d} \oplus R_{d+1} \oplus \cdots$. 
% This degree $d$ will equal $\max(\post(R), \post(R^u))$. To simplify notation, let  $\nz(R) = max(\post(R)), \post(R^u))$ by. When to context is clear, we
% will omit $R$ and only write $\nz$. For instance, we will write $R_{\nz}$ instead of $R_{\nz(R)}$.

\section{Projective multiplication matrices}

We will use the fact that rings of projective dimension zero postulates in order to define the projective multiplication matrices.

\subsection{Non-zero divisors of degree one}
If $p_i = (p_{i0}: \cdots : p_{in})$ is a projective point with respect to the coordinates $x_0, \ldots, x_n$, then we cannot define evaluation on a form in a unique way since 
$\lambda p_i = p_i$ for non-zero $\lambda$.
For practical purposes it is however good to define  the evaluation in a unique way and we will do this by simply fixing a representation of each projective point. So we associate to each projective point $p_i = (p_{i0}: \cdots : p_{in})$ an affine point $p_i^a = (p_{i0}, \ldots, p_{in})$ and now we can define the unique evaluation as $x_0^{\alpha_0} \cdots x_n^{\alpha_n}(p_i) = x_0^{\alpha_0} \cdots x_n^{\alpha_n}(p_i^a) = p_{i0}^{\alpha_0} \cdots p_{in}^{\alpha_n}$. This way of thinking of evaluation is implicit in \cite{robbiano}.
 
With this evaluation method it follows that if $l_1$ and $l_2$ are elements of $S$ and $[l_1] = [l_2]$ in $R$, then 
$l_1(p) = l_2(p)$ for all points $p$ on $V(I)$. This property makes it possible to define evaluation on elements in $R$ by $[l](p) = l(p)$, where $p \in V(I)$. 

We will need a notation for evaluating a set of points on a set of elements. If $P= \{p_1, \ldots,p_m\}$ is a set of projective points, we write $f(P) = (f(p_1), \ldots, f(p_m))$. If $F =
\{f_1, \ldots, f_s\}$ is a set of forms in $S$, then $F(P)$ is defined to be the $(s \times m)$-matrix whose $i$'th row is $f_i(P)$.

\begin{lemma} \label{lemma:nonzeroeval}
 Let $I$ be an ideal of projective dimension zero. Suppose that $\V(I) = \overline{\V(I)}$.  
That $[l]$ is a non-zero divisor on $R^u$ is equivalent to $l(p)\neq 0$ for all $p \in V(I)$.
\end{lemma}
\begin{proof}
Suppose that $l(p) \neq 0$ for all $p \in V(I)$. If $[a] \in R^u$ is such that  
$[a] \cdot [l] = 0$ in $R^u$, then 
$a \cdot l \in I^u$, so that $(a \cdot l) (p) = a(p) \cdot l(p) = 0$ for all $p \in V(I)$. Hence $a(p) = 0$ for all $p \in V(I)$.
Thus $a \in I$, so $[l]$ is a non-zero divisor on $R^u$.

Suppose instead that $[l]$ is a non-zero divisor on $R^u$. Let $q$ be an arbitrary point in $V(I)$. 
Let $[Q]$ be an element in $R^u$ such that $Q(q) \neq 0$ and $Q(p) = 0$ for $p \in V(I) \setminus \{q\}$. 
(The element $Q$ is called a separator for the point $q$ with respect to $V(I)$. Separators exists, see for instance \cite{Abbottetal} or Section \ref{sec:computingseparators}.)
Suppose that $l(q) = 0$. Then $Q(p) \cdot l(p) = 0$ for all $p \in V(I)$, so that $[Q] \cdot [l] = 0$. Since $[Q] \neq 0$, it follows that $[l]$ is a zero-divisor, which is a contradiction. Hence $l(p) \neq 0$ for all $p \in V(I)$. 
 \end{proof}
 
\begin{proposition}  \label{prop:nonzerodivisordegone}
Suppose that $\Bbbk$ contains at least $|\overline{V(I)}|$ elements. 
Then $R$ has a linear non-zero divisor if and only if $I$ is unmixed.
The requirement on $\Bbbk$ is sharp in the sense that if $\Bbbk$ contains $|\overline{V(I)}|-1$ elements, 
then there exists an unmixed ideal $I$ such that $R$ lacks linear non-zero divisors.
\end{proposition}

\begin{proof}
If $I$ is mixed, then $R$ does not contain non-zero divisors. So suppose that $I$ is unmixed. Let $s = |\overline{V(I)}|$ and let $I = \overline{Q_1} \cap \cdots \cap \overline{Q_s}$ be a primary decomposition with respect to $\overline{S}$. Let 
$\overline{P_i} = \sqrt{\overline{Q_i}}$. Each $\overline{P_i}$ is generated in degree one.
The set  of zero-divisors in $\overline{R}$ equals the union of the residues of the $\overline{P_i}$:s (Proposition 4.7 in \cite{Atiyah}). 
So if we let $\overline{\NZ} = \overline{S_1} \setminus \cup \overline{P_i}$, then the set of linear non-zero divisors in $\overline{R}$ is the residues of $\overline{\NZ}$.
Let $\NZ = S_1 \cap \overline{NZ}$. Then the linear non-zero divisors of $R$ is the residues of $\NZ$. Let $\eta_i = \overline{P_i} \cap S_1$.
% Suppose first that $\Bbbk$ is infinite. Then, if no $\mathfrak{p_i}$ equals $S_1$, then their union (which is finite) can not equal $R_1$. 
% It follows that $S/I$ contains a nonzero-divisor if and only if no $P_i$ equals $(x_0, \ldots, x_n)$, that is, 
% if and only if the maximal ideal is not associated to $I$.

Suppose that $\Bbbk$ is infinite or finite and contains $|\Bbbk| \geq s$ elements.
Suppose that $v_i \in S_1$ but $v_i \notin \eta_{1} \cup \cdots \cup \eta_i$ for $i<s$ (clearly $v_1$ exists). 
If $v_i \notin \eta_{i+1}$, then let $v_{i+1} = v_i$. Otherwise, take an element $w_i \in \eta_i$ such that $w_i \notin \eta_{i+1}$ (such an element must exist, since we assume that
$v_i \notin \eta_{i+1}$). The element $v_i + \alpha w_i$ does neither belong to $\eta_i$ nor to $\eta_{i+1}$ for any non-zero 
$\alpha \in \Bbbk$. Pick $\alpha_1 \in \Bbbk \setminus \{0\}$. 
If $v_i + \alpha_1 w_i \in \eta_{i-1}$, then $v_i + \alpha w_i \notin \eta_{i-1}$ for all $\alpha \in \Bbbk \setminus \{\alpha_1\}$, since otherwise we would have
$\alpha (v_i + \alpha_1 w_i) -  \alpha_1 (v_i + \alpha w_i) = (\alpha - \alpha_1)v_i \in \eta_{i-1}$, which is a contradiction. So pick 
$\alpha_2  \in \Bbbk \setminus \{\alpha_1\}$. Clearly $\eta_i + \alpha_2 w_i \notin \eta_{i-1} \cup \eta_{i} \cup \eta_{i+1}$. It is clear that we can continue
in this way provided that there is at most $i-1$ non-zero elements in $\Bbbk$. Since $i$ ranges from $1$ to $s$, this construction uses at most 
$s-1$ non-zero elements in $\Bbbk$. 

Suppose instead that $\Bbbk$ is a field with elements $\{a_0, a_1, a_2, \ldots, a_{s-1}\}$ with $a_0 = 0$ and $a_1 = 1$. 
Consider the points $p_1 = (1:0:\cdots:0), p_2 = (1:1:0:\cdots:0), p_3 = (1:a_2:0:\cdots:0), \ldots, p_s =  (1:a_{s-1}:0:\cdots:0)$ and
$p_{s+1} = (0:1:0:\cdots:0)$ in $P^n(\Bbbk)$. Let $I(p_i)$ be the vanishing ideal with respect to $p_i$, which is prime. Let
$I = \cap_i I(p_i)$ and let $l = b_0x_0 + \cdots + b_nx_n$ be an arbitrary linear form. 
If $l$ is non-zero on $p_1, \ldots, p_s$, then 
$b_0  + b_1 a_0 \neq 0$, $b_0 + b_1 a_1 \neq 0, \ldots, b_0 + b_1 a_{s-1} \neq 0$. Now $b_1$ must equal zero, since otherwise we would have
$b_0 + b_1 a_i \neq b_0 + b_1 a_j$ when $i \neq j$ and thus, by the pigeonhole principle,  $b_0 + b_1a_i = 0$ for some $i$, which contradicts the assumption that 
$l(p_i) \neq 0$. But if $b_1 = 0$, then $l(p_{s+1}) = 0$. Thus, the ring lacks linear non-zero divisors by Lemma \ref{lemma:nonzeroeval}. 

\end{proof}

%\begin{remark}
%x^2 - 2y^2, x^2-2z^2 in Z_3[x,y,z]
%At a first glance, it might seem strange to consider each $P_i$ as an element of $\overline{S}$ and then construct $P_i \cap S_1$. One might think that if
%$P_i$
%It is important to consider the primary decomposition in the algebraic closure of $\Bbbk$. $P_i$

%\end{remark}

\begin{example} \label{ex:propnonzero}

Let $I = (x_0+x_1,x_0+x_2) \cap (x_0+x_1,x_0+2x_2) \cap (x_0+2x_1,x_0+2x_2) \subseteq Z_3[x_0,x_1,x_2]$. We have $V(I) = \{(1:2:2), (1:2:1), (1:1:1)\}$. To compute a non-zero divisor, we start by 
computing $v_1$. We pick an element of degree one in $(x_0+x_1,x_0+x_2)$, say $x_0 + x_1$. By changing one of the coefficients, we can assure that this element is not in $(x_0+x_1,x_0+x_2)$, so we let $v_1 = x_0+2x_1$. Since $v_1(p_2) \neq 0$ we let $v_2 = v_1$. But $v_2(p_3) = 0$, so we look for an element $w_2$ which is in $(x_0+x_1,x_0+2x_2)$ but not in 
$(x_0+2x_1,x_0+2x_2)$. 
It is clear that we can find such an element by going through the generators of $(x_0+x_1,x_0+2x_2)$ until we find an element which is not in $p_3$. Indeed, $w_2 = x_0+x_1$ is such an element. We have $v_2 + w_2 = 2x_0$. Since $2x_0(p_1) \neq 0$ we can use $[2x_0]$ (or rather $[x_0]$) as a non-zero divisor.

\end{example}

When $\Bbbk$ is finite, it is an interesting question to determine, given a degree $d$, the maximal number of points allowed to guarantee the existence of a non-zero divisor of degree $d$. Using a result due to Chevalley, one can show that there is always a non-zero divisor of degree $n$, even if all points in $P^n(\Bbbk)$ are considered. However, we will only use non-zero divisors of degree one in this paper and this problem will be dealt with in a separate paper \cite{lundqvist3}. 

%It is hard to give a similiar criteria to garantee the existence of a non-zero divisor of degree two, three and so on. 
%We will now use Chevalley's theorem to show that there does always exists a %non-zero divisor of degree $n$. 

%\begin{theorem}[Chevalley]
%Let $\Bbbk$ be a finite field and let $f$ be a form in $S$. 
%If $n>d$ then there exists a projective point $p$ such that $f(p) = 0$.
%\end{theorem}

%Chevalley's theorem shows that if we consider the vanishing ideal with respect %to all projective points 
%in $P^{n}(\Bbbk)$, then there is no non-zero divisor of degree less than $n$ in %$R$. Indeed, if $f$ is any form of degree less than $n$, then there is
%a point $p$ such that $f \in I(P)$, and hence $f$ is in the union of the prime %ideals of $I$. 

%\begin{theorem}

%\end{theorem}

%Together these results show 

%\begin{proposition}
%Let $\Bbbk$ be a finite field and let $I$ be an unmixed ideal of projective %dimension zero in $S$. 
%Then there is always a non-zero divisor of degree $n$ and it equals the $\ldots$
%\end{proposition}

\begin{proposition}\label {prop:onto}
Let $I$ be an ideal of projective dimension zero. Suppose that $\Bbbk$ contains at least $\overline{|V(I)|}$ elements.
Then there exists a linear form $l \in S_1$ such that $\mathfrak{L}: R_d \to R_{d+1}, [a] \mapsto [l][a]$ is onto, for all $d \geq \nz(R)$.
\end{proposition}

\begin{proof}
If $R = R^u$, then, by Proposition \ref{prop:nonzerodivisordegone}, $R$ contains a nonzero-divisor of degree one which has the desired property. 
Otherwise, the maximal ideal is associated to $I$. Thus, the primary decomposition of $I$ can be written as  $I = J \cap Q$, with 
$\sqrt{Q} = \mathfrak{m}$. Let $d \geq \nz(R)$. 
Then $\dim_{\Bbbk}(R_d) = \dim_{\Bbbk}( R^{u}_d)$, which is equivalent to $\dim_{\Bbbk}(J_d) = \dim_{\Bbbk}(J \cap Q)_d$. 
Since $(J \cap Q)_d \subseteq J_d$, it follows that $J_d = (J \cap Q)_d$. So that $l$ is a nonzero-divisor on  $\oplus_{i\geq d} R^{u}_d$
implies that $l$ is a nonzero-divisor on $\oplus_{i\geq d} R_d$. 
\end{proof}

Suppose that $\{[e_1], \ldots, [e_m]\}$ is a basis for the $\Bbbk$-space $R_d, d \geq \nz(R)$. Since the map induced by $l$ is onto, the set 
$\{[e_1l], \ldots, [e_ml]\}$ forms a $\Bbbk$-basis for $R_{d+1}$. In general, for any positive integer 
$i, \{[e_1 l^i], \ldots, [e_ml^i]\}$ can be chosen as $\Bbbk$-basis for $R_{d+i}$. 
This shows the following important theorem.

\begin{theorem} \label{thm:basis}
Let $I$ be an ideal of projective dimension zero. Let $\{e_1, \ldots, e_m\}$ be a $\Bbbk$-basis
for $R_{\nz}$. 
Suppose that $R^{u}$ has a non-zero divisor $[l]$ of degree one. Then
$\{[e_1l^i], \ldots, [e_m l^i]\}$ is a $\Bbbk$-basis for  $R_{\nz + i}$, 
for all positive integers $i$.
The matrix multiplication matrix $A_j$ representing the map $R_{\nz + i} \to R_{\nz +i+1}$, $[a] \mapsto [x_j a]$ with respect to the bases above, is independent of the choice of $i$. 
\end{theorem}

\subsection{An affine connection}

To a ring $R$ of projective dimension zero, we will now associate an affine ring of dimension zero --- $R_*$, whose
multiplication matrices coincide with the projective multiplication matrices of the projective ring. 
In fact, the zero-dimensional ring $R_*$ is simply $R/([l]-[1])$. The key is the following lemma.
%We state it as a lemma although it is by no means new and is used for instance by \cite{Marinariproj} to develop a technique to %determine a Gröbner basis for a 1-dimensional ideal by means of the Gröbner basis of a zero dimensional one.

\begin{lemma} \label{lemmalinear}
Suppose that $\Bbbk$ contains at least $|\overline{V(I)}|$ elements. Then there is a linear change of coordinates $T$ and a variable $x_i$ such that $T(x_i)(p) \neq 0$ for all points $p \in V(I)$. 
\end{lemma}
\begin{proof}
Let $l$ be the form from proposition \ref{prop:onto}. We can write $l = b_0x_0 + \cdots + b_nx_n$. Some coefficient is non-zero, say $b_i \neq 0$. 
Let $T(x_i) = l$ and let $T(x_j) = x_j$ if $j \neq i$.
\end{proof}

%\begin{remark} \label{rem:firstcoord}
%If $R^{u}$ contains a non-zero divisor of degree one, then we may, after a %linear change of coorindates and without loss of generality, assume that $x_0$ %is the nonzero-divisor. Accordingly, 
%that every point $p$ in the variety can be written as $(1:a_1:\cdots:a_n)$, for %some $a_i$. Thus we may assume that $\V(I)$ contains no points at infinity.
%\end{remark}

\begin{remark} \label{rem:firstcoord}
With $T$ as above, if we let 
$y_0 = T(x_i), y_1 = T(x_0), y_2 = T(x_1), \ldots,$ 
$y_{i} =T(x_{i-1}), y_{i+1}=T(x_{i+1}), \ldots, y_{n} = T(x_n)$, then each point on $V(I)$ with respect to $y_0, \ldots, y_n$ can be written as 
$(1:a_1:\cdots:a_n)$.
\end{remark}

Let $J$ be an ideal of affine dimension zero and let
$J=q_1 \cap \cdots \cap q_s$ be a minimal primary decomposition. The multiplicity of a 
point $p \in V(J)$, belonging to the primary component $\sqrt{q_i}$, is defined as the length of $q_i$.

The same definition holds for projective points, that is,  if $I$ is an ideal of projective dimension zero and
$I =  q_1 \cap \cdots \cap q_s$, then the multiplicity of a point $p$ on $V(I)$ belonging to the primary component $\sqrt{q_i}$, is defined as the length of $q_i$. 

These two multiplicity definitions are connected in the sense that if $p=(1:a_1:\cdots:a_n)$ is a projective point with multiplicity $r$ in $V(I)$, then $(a_1,\ldots,a_n)$ is an affine point with multiplicity $r$ in $V(I+(y_0-1))$. This is a standard result and treated in \cite{Grobner} and \cite{MMM} for instance. 

\begin{lemma} \label{lemma:kbasis}
Let $I$ be an ideal of projective dimension zero. Suppose that there exists an
$l \in S_1$ such that $[l]$ is a non-zero divisor on $R^{u}$.
Put $R_* = R/([l]-[1])$. Let $\{[e_1], \ldots, [e_m]\}$ be a $\Bbbk$-basis for
 $R_d$, for $d \geq \nz(R)$. Then $\{[[e_1]], \ldots, [[e_m]]\}$ is a $\Bbbk$-basis for $R_*$, where $[[\hspace{0.2cm}]]$ denotes an equivalence class in $R$ mod $[l]-[1]$. 
\end{lemma}

\begin{proof}
Since the $\Bbbk$-dimension of $R/([l]-[1])$ and $R_{\nz}$ is determined by the sum of the points counting multiplicity, we have
 $\dim_{\Bbbk}( R/([l]-[1])) = \dim_{\Bbbk}(R_{\nz})=m$.
 % by Lemma \ref{lemma:onetoone}. 
 Since $[l]$ is a non-zero divisor on the basis $R_{\nz}$, the vectors $[[e_1]], \ldots, [[e_m]]$ are linearly independent.
\end{proof}

\begin{proposition} \label{prop:onetozero}
Let $I$ be an ideal of projective dimension zero. Suppose that $\Bbbk$ contains at least $|\overline{V(I)}|$ elements.
Let $m = \dim_{\Bbbk}(R_{\nz})$. Put $R_* = R/([l]-[1]) = S/(I + ([l]-[1]))$, where $l$ is a non-zero divisor of degree one on $R^u$. Let $\{[e_1 l^i], \ldots, [e_m l^i]\}$ be a $\Bbbk$-basis for
 $R_{d+i}$ and let $\{[[e_1]], \ldots, [[e_m]]\}$ be a $\Bbbk$-basis for $R_*$, where $[[\hspace{0.2cm}]]$ denotes an equivalence class in $R$ mod
 $[l-1]$.  Then there is a change of coordinates such that the multiplication matrices with respect to 
$x_1, \ldots, x_n$ coincides for $R_*$ and $R$.
\end{proposition}

\begin{proof}

By a change of coordinates in accordance with Remark \ref{rem:firstcoord}, we may assume that $l = y_0$.
Let $A_k = (a_{ij})$ be the projective multiplication matrix of $R$ with respect to $y_k$, such that $[y_k] [e_i] =
 a_{i1} [y_0 e_1] + \cdots +a_{im} [y_0 e_m]$.  It follows that $[[y_k]] [[e_i]] =
 a_{i1} [[e_1]] + \cdots +a_{im} [[e_m]]$. By Lemma \ref{lemma:kbasis}, $\{[[e_1]], \ldots, [[e_m]]\}$ is
a $\Bbbk$-basis for $R_*$. Hence $A_k$ is the multiplication matrix of $R_*$ with respect to $y_k$. 

\end{proof}

\subsection{Computing the variety from the projective multiplication matrices}

%Theorem \ref{thmonetozero} is crucial, since it gives a link between the theory of the graded multiplication matrices and the classical multiplication matrices. 

%Let $A_1, \ldots, A_n$ be a collection of square matrices. 
%When $A$ is a square matrix, let $E(\lambda, A)$ denote the eigenspace to the eigenvalue $\lambda$ of %%$A$. 

We now state the result of M\"oller and Stetter in the affine setting. See \cite{Mollerstetter} for a proof. 
\begin{theorem} \label{thm:mainaffine}
Let $I$ be a zero dimensional ideal. Let $\{[e_1], \ldots, [e_m]\}$ be a $\Bbbk$-basis of $\Bbbk[x_1, \ldots, x_n]/I$. 
Let $A_1, \ldots, A_n$ be the multiplication matrices with respect to this basis. 
Let $r = |V(I)|$. Then there are exactly $r$ common (right) eigenvectors for the matrices $A_1, \ldots, A_n$ and they are 
$(e_1(p_i), \ldots, e_m(p_i))^t$ for $ i = 1, \ldots, r$.  Let $\lambda_{ij}$ denote the eigenvalue of $A_j$ corresponding to the eigenvector 
$(e_1(p_i), \ldots, e_m(p_i))^t.$ Then $p_i = (\lambda_{i1}, \ldots, \lambda_{im}).$ 
\end{theorem}
\noindent
We have an almost identical theorem in the projective setting.
\begin{theorem} \label{thm:mainproj}
Let $I \subseteq \Bbbk[x_0, \ldots, x_n]$ be an ideal of projective dimension zero. Suppose that
$\Bbbk$ contains at least $|V(I)|$ elements. 
Let $\{[e_1], \ldots, [e_m]\}$ be a $\Bbbk$-basis of $R_{\nz}$ and let $l$ be a linear form such that $\{[e_1l], \ldots, [e_ml]\}$ is a $\Bbbk$-basis for
$R_{\nz +1}$.  Let $A_0, \ldots, A_n$ be the projective multiplication matrices with respect to this basis.
Let $r = |V(I)|$. Then there are exactly $r$ common (right) eigenvectors for the matrices $A_0, \ldots, A_n$ and they are 
$(e_1(p_i), \ldots, e_m(p_i))^t$ for $ i = 1, \ldots, r$.  Let $\lambda_{ij}$ denote the eigenvalue of $A_j$ corresponding to the eigenvector 
$(e_1(p_i), \ldots, e_m(p_i))^t.$ Then $p_i = (\lambda_{i0}:\lambda_{i1}: \cdots: \lambda_{im}).$ 
\end{theorem}

\begin{proof}
By Proposition \ref{prop:onto}, there exists a linear form $l = b_0x_0 + \cdots + b_nx_n$ such that $\{[e_1 l], \ldots, [e_m l]\}$ forms a $\Bbbk$-basis
for $R_{\nz+1}$. By Remark \ref{rem:firstcoord}, there is a change of coordinates such that the multiplication matrix $B_i$ with respect to $y_i$
satisfies $B_0 = b_0A_0 + \cdots b_nA_n$ and $B_1 = A_0, B_2 = A_1, \ldots, B_i = A_{i-1}, B_{i+1} = A_{i+1}, \ldots, B_{n} = A_n$.

By Proposition \ref{prop:onetozero}, the projective multiplication matrices $B_1, \ldots, B_m$ of $R$ with respect to $y_1, \ldots, y_m$ 
agree with the multiplication matrices for $R/(l-1)$ with respect to $y_1, \ldots, y_m$.

Since the multiplication matrix with respect to $y_0$ is the identity, a common eigenvector for $B_1, \ldots, B_n$ 
is also a common eigenvector for $B_0, \ldots, B_n$ and vice versa. But by linearity, $v$ is a common eigenvector to 
$B_0, \ldots, B_n$ if and only if $v$ is an eigenvector to $A_0, \ldots, A_n$. 
Hence, the set of common eigenvectors for $A_0,  \ldots, A_n$ equals 
$e_1(p_i), \ldots, e_m(p_i)$, for $i = 1, \ldots, r$, by Theorem \ref{thm:mainaffine}.
\end{proof}

To determine the multiplicity of a point $p \in V(I)$, one can use the result of Corless et al in \cite{corless}. The method goes as follows.
Let $A$ be a generic linear combination of the multiplication matrices. Let $\lambda$ be the eigenvalue of $A$ with respect to $e(p)$ (clearly $e(p)$ is an eigenvector
of $A$). Then the multiplicity of $p$ equals the algebraic multiplicity of $\lambda$.
There are also direct methods which one could use, see for instance \cite{MMMtrans} and \cite{Mollertenberg}.

%
%\begin{example} \label{example:mainnew}
%Consider the ideal vanishing on the six points ( ). This is example 3.5 from \cite{Abbotetal}.

%$f_1 = 68 x(3)^3-149 x(1) x(3)^2+88 x(2) x(3)^2-36 x(2)^2 x(3)+47 x(1) 
%x(2) x(3),
%f_2 = 28x(3)^3-183 x(1) x(3)^2+180 x(2) x(3)^2+47 x(1)^2 x(3)-48 x(2)^2 
%x(3),
%f_3 = 752 x(1)^3-11468
%    x(2)^2 x(1)+99381 x(3)^2 x(1)+16920 x(2)^3-110528 x(3)^3+55395 x(2) 
%x(3)^2-56602
%    x(2)^2 x(3),
%    f_4 = 1880 x(2)^3-1692 x(1) x(2)^2-5594 x(3) x(2)^2+376 x(1)^2 
%x(2)+4739 x(3)^2
%    x(2)-12192 x(3)^3+11661 x(1) x(3)^2\right\}

\begin{example} \label{example:main}
The elements
\begin{alignat*}{2}
f_1& = xz+yz-z^2\\
f_2& =x^2-y^2+2yz-z²\\
f_3& = xy-y²+yz
\end{alignat*}
generates an unmixed ideal $I$ of projective dimension zero in $\mathbb{C}[x,y,z]$.
Choosing ${[x],[y],[z]}$ and ${[y^2], [yz], [z²]}$ as bases in degree $1$ and $2$ respectively, we see that neither $[x]$, $[y]$ nor $[z]$ serve as nonzero-divisors. Indeed, if we let $M_x, M_y$ and $M_z$ denote the multiplication matrices from $R_1$ to $R_2$ with respect to the bases chosen above, we compute

$$M_x = 
\left(
\begin{array}{rrr}
1 & -2 & 1 \\
 1 & -1 & 0 \\
 0 & -1 & 1
\end{array}
\right),
M_y = 
\left(
\begin{array}{rrr}
1 & -1 & 0 \\
 1 & 0 & 0 \\
 0 & 1 & 0
\end{array}
\right), 
M_z = 
\left(
\begin{array}{rrr}
0 & -1 & 1 \\
 0 & 1 & 0 \\
 0 & 0 & 1
\end{array}
\right)
$$
and we can see that all the matrices have a nontrivial kernel. However, $M_y + M_z$ has full rank which is equivalent 
to $[y+z]$ being a nonzero-divisor. 
Hence, if we use $\{[x(y+z)], [y(y+z)], [z(y+z)]\}$ as a $\Bbbk$-basis in degree two, we can construct the projective multiplication matrices 
$A_x, A_y$ and $A_z$. From these matrices the solutions can be read off. Now 
$[x(y+z)] = [y^2]-2 [z y]+[z^2], [y(y+z)] = [y^2]+[zy], [z(y+z)]=[yz]+ [z^2]$ by making use of the multiplication matrices above. Thus, with  
$$T = 
\left(
\begin{array}{rrr}
1& 1& 0\\
-2&1& 1\\ 
1 &0& 1
\end{array}
\right)
$$ we have
%transforms the second basis into the first. 

%We invert this matrix and get 
%$T =  \frac{1}{4} \begin{pmatrix} 
% 1 & -1 & 1 \\
% 3 & 1 & -1 \\
% -1 & 1 & 3
%\end{pmatrix}
%$ 
%Finally, 
$A_x = M_x (T^t)^{-1}$ and similarly for $A_y$ and $A_z$, so that$$
A_x = \frac{1}{2}
\left(
\begin{array}{rrr}
 2 & 0 & 0 \\
 1 & 1 & -1 \\
 1 & -1 & 1
\end{array}
\right),
A_y = \frac{1}{4}
\left(
\begin{array}{rrr}
 2 & 2 & -2 \\
 1 & 3 & -1 \\
 -1 & 1 & 1
\end{array}
\right)
,$$
and $$
A_z = \frac{1}{4}
\left(
\begin{array}{rrr}
 2 & -2 & 2 \\
 -1 & 1 & 1 \\
 1 & -1 & 3
\end{array}
\right)
.$$

Common eigenvectors for the matrices are $(1,1,0), (1,0,1)$ and $(0,1,1)$. The eigenvalues corresponding to 
$(1,1,0)$ are $1,1,0$ for $A_x, A_y$ and $A_z$ respectively. Likewise, the eigenvalues corresponding to $(1,0,1)$ are $1,0,1$ 
and the eigenvalues corresponding to $(0,1,1)$ are $0,1,1$. Thus, $V(I) =  \{(1:1:0), (1:0:1), (0:1:1)\}$.

Notice that since $\nz(R) = 1$, we can also use the correspondence between eigenvectors and the $\Bbbk$-basis to obtain the points. Indeed 
$([x](p_1),[y](p_1),[z](p_1)) $ $= (1,1,0)$, thus we have $p_1 = (1:1:0)$, etc.

\end{example}

\section{Applications and computational aspects}

A convenient way to think of a ring $R=S/I$ of projective dimension zero is as
$$\tilde{R} = R_0 \oplus R_1 \oplus \cdots \oplus R_{nz}$$ together with the linear map
$l$ and the multiplication matrices $A_1, \ldots, A_n$. We write this information as a triplet $(\tilde{R}, A, l)$. 
The $\Bbbk$-dimension of the graded pieces of $\tilde{R}$ describes the configuration of the points and also tells whether or not the maximal ideal is associated, while  
the multiplication matrices encode the variety as a set. In Section \ref{sec:nf} we will see that we obtain a fast normal form algorithm by using the triplet. With this perspective, 
the classical way of determining a Gr\"obner basis for $I$ 
misses a lot of information about the ring. It also turns out that we compute
unnecessary data. For instance, a Gr\"obner basis for the ideal $I = (xz + yz -z^2, x^2-y^2+2yz-z^2,xy-y^2+yz)$ 
from Example \ref{example:main} with respect to $x>y>z$ and DegRevLex is $(xz + yz -z^2, x^2-y^2+2yz-z^2,xy-y^2+yz, y^2z-yz^2)$. 
Since $\nz(R) = 1$, we only need to consider the $\Bbbk$-spaces $R_1$ and $R_2$ to determine the variety, and for this purpose, the term $y^2z-yz^2$ in the Gr\"obner basis is superfluous. In Section \ref{sec:gb} we will show that the maximal degree of a term in a Gr\"obner basis is 
$\max(\nz(R),m)$.  Since it is enough to compute up to degree $\nz(R)$ in order to determine the variety, this
indicates that Gr\"obner techniques are not always optimal. 
Unfortunately, it is hard to detect $\nz(R)$.

\subsection{Computing normal forms with respect to $(\tilde{R},A,l)$} \label{sec:nf}

As an application of the multiplication matrices, we obtain a fast 
normal form algorithm for high degree elements of $S$. 
Suppose that we have a normal form algorithm $\nf(*,B)$ for elements of degree less than or equal to $\nz(R)$.
To extend this method to elements of degree $> \nz(R)$, we proceed as follows.
Let $a \cdot b$ be a monomial in $S$ and suppose that $|b| = \nz(R)$.  We use the normal form algorithm for low degree elements to obtain 
$\nf(b,B) = b_1 e_1 + \cdots + b_m e_m$. To determine $\nf(ab,B)$, write 
$a = x_1^{a_1} \cdots x_n^{a_n}$. It is straightforward to check that 
$$\nf(ab,B) = (b_1, \ldots, b_m) A_1^{a_1} \cdots A_n^{a_n} (l^{|a|} e_1, \ldots, l^{|a|} e_m)^t.$$ Thus, the arithmetic complexity of 
the normal form algorithm is $O(|a|m^3)$ if one uses naive matrix multiplication or 
$O(|a|m^{2.376})$ if one uses state of the art methods \cite{matrix}. To this one needs to add 
the complexity for computing $\nf(b,B)$. 
%Thus, the computation of the normal form of the high degree part of the monomial has polynomial complexity. 

\begin{example}

Suppose that we want to compute the normal form of $x^{17}$ with respect to the ideal $I$ from Example \ref{example:main}.
We have seen that $\{[x(y+z)^i], [y(y+z)^i], [z(y+z)^i]\}$ forms a $\Bbbk$-basis for $R/I$ and that 
$$A_x = \frac{1}{2}
\left(
\begin{array}{rrr}
 2 & 0 & 0 \\
 1 & 1 & -1 \\
 1 & -1 & 1
\end{array}
\right).
$$

Since $[x]$ is a basis element in degree one, the normal form of $x^{17}$ equals 
$$(1,0,0) A^{16} (x(y+z)^{16}, y(y+z)^{16}, z(y+z)^{16})^t.$$ Since $A^{2} = 2 \cdot A$, we have
$A^{16} = 2^{15}A$. Hence $$\nf(x^{17}, x(y+z)^{16}, y(y+z)^{16}, z(y+z)^{16}) = 2^{15} x(y+z)^{16}.$$ 
\end{example}
If we know the variety of $I$, then the normal form computation can be simplified, see Example \ref{ex:multalg}.

\subsection{Upper bound of the elements in a Gr\"obner basis} \label{sec:gb}

To give an upper bound of the maximal degree of an element in a Gr\"obner basis with respect to an ideal of projective dimension zero, we will use 
Gotzmann's persistence theorem.

Recall that if $h$ and $i$ are positive integers, then $h$ can be uniquely written as a sum
$$h = \binom{n_i}{i} + \binom{n_{i - 1}}{i-1} + \cdots + \binom{n_j}{j},$$
where 
$$n_i > n_{i-1} > \cdots > n_j \geq j \geq 1.$$ 
See \cite{robbiano} for an easy proof.
This sum is called the binomial expansion of $h$ in base $i$.
Define 
$$h^{<i>} = \binom{n_i + 1}{i + 1} + \binom{n_{i-1} + 1}{i}+ \cdots +
\binom{n_j + 1}{j+1}.$$

Before stating Gotzmann's theorem, recall that 
the Hilbert function of a graded algebra $R$ is the map $d \mapsto \dim_{\Bbbk}(R_d)$.

\begin{theorem}[Gotzmann's persistence theorem \cite{gotzmann}]
Let $\Hf$ be the Hilbert function of $k[x_1, \ldots, x_n] / I$, 
for any homogeneous ideal $I$. Let $t$ denote the 
maximal degree of the generators of $I$. Then $\Hf(d+1) = \Hf(d)^{<d>}$ for some $d \geq t$ implies that 
$\Hf(d+2) = \Hf(d+1)^{<d+1>}, \Hf(d+3) = \Hf(d+2)^{<d+2>} $ and so on.
\end{theorem}
In the case of projective points, we have 
$\Hf(d+1) =\Hf(d) = m$ when $d \geq\nz(R)$, thus we have
$\Hf(d+1)^{<d+1>} = \Hf(d)^{<d>} = m$ for $d \geq \nz(R)$ and hence
$$\Hf(d)^{<d>} = \binom{d+1}{d+1} + \cdots + \binom{d-(m-2)}{d-(m-2)}.$$ 

A Lex-segment set $L_d$ on $\{x_1, \ldots, x_n\}$ is
the $|L_d|$ biggest monomials of degree $d$ in $\Bbbk[x_1, \ldots, x_n]$ 
with respect to the lexicographical ordering. When $L$ is a collection of Lex-segment sets, let 
$I(L)$ denote the ideal generated by the elements in the Lex-segment sets. We call $I(L)$ a Lex-segment ideal.
When $I$ is a homogeneous ideal, let $|\ini(I)^c_d|$ denote the number of monomials outside 
$\ini(I)$ of degree $d$. Notice that $|\ini(I)^c_d|$ is independent of monomial ordering.

Let $I$ be a homogeneous ideal generated in degree less than or equal to $d$ and let $L$ be a collection of 
Lex-segment sets with maximal degree $d$. A property among Lex-segment ideals is that they have minimal growth (or maximal co-growth), in the sense 
that if $|\ini(I)^c|$ and $|\ini(I(L))^c|$ agrees until degree $d$, then 
$$|\ini(I(L))^c_{d'}| \geq |\ini(I_{\leq d})^c_{d'}| \text{ for all } d'\geq d.$$ See for instance \cite{robbiano}. 

\begin{theorem} \label{thm:maximaldegree}
Let $I$ be an ideal of projective dimension zero. A bound for the maximal degree of an element in a reduced Gröbner basis is $\max(\nz(R), m)$.
\end{theorem}

\begin{proof}
Let $d = \max(\nz(R), m)$. Suppose that $L$ is a collection of Lex-segment sets of degrees less than or equal to $d$, 
such that $|\ini(I(L))^c|$ agrees with $|\ini(I)^c|$ until degree $d$. We then have
$$ m = |\ini(I(L))_{d'}^c| \geq |\ini(I_{\leq d})_{d'}^c| \geq |\ini(I_{\leq d'})_{d'}^c| = m.$$
This implies that 
$|\ini(I_{\leq d})_{d'}^c| = |\ini(I_{\leq d'})_{d'}^c|$ for all $d'\geq d$ 
and hence there can not be any Gr\"obner basis element of degree greater than $d$. 
%$\ini(I)$ is generated in degrees less than or equal to $d$.
\end{proof}
This theorem is a generalization of the result in \cite{Abbottetal}, 
where it is shown that the last degree element of a Gr\"{o}bner basis is $m$ in the case when $I$ is unmixed. 
The bound in Theorem \ref{thm:maximaldegree} is sharp. Indeed, 
in Example \ref{example:main}, $\nz(R) = 1, m=3$ and a reduced Gr\"obner basis with respect to DegRevLex had a generator in degree three, while in 
Example \ref{example:localnonzerodivisor} below, we will see that $\nz(R) = 3, m=1$ and a reduced Gr\"obner basis with respect to DegRevLex is 
$\{xy - z^2, x^2-xz, y^2 - z^2, xz^2 - yz^2, -yz^2 + z^3\}$.

%Let $I=(x_0^2, x_1^2-x_0x_n, x_2^2-x_1x_n, \ldots, x_{n-1}^2-x_{n-2}x_n, %x_{n-1}x_n).$ Then $V(I)= \{(0:\cdots:0:1)\}$, so $m=1$ while $\nz(R)=n$. A %reduced Gröbner basis with respect to DegRevlex is ll  and V(I)= and a Gröbner %basis with respect to c

\subsection{Computing $(\tilde{R},A,l)$ given the ideal}

Suppose that we are given an ideal by its generators and that we know that $\dim_{\Bbbk}(R_d) = \dim_{\Bbbk}(R_{d+1})$ for some $d$. What conclusions can be made from this information?
Unfortunately, not many. We do not know the dimension --- indeed --- the rings 
$\Bbbk[x,y,z]/(xy,yz,xz), \Bbbk[x,y,z]/(x^2,y^2,z^2)$ and $\Bbbk[x,y,z]/(x^2,xy,xz)$ all have 
$\Bbbk$-dimension three in degrees one and two. The first ring is of projective dimension zero and postulates in degree one. The second ring
is artinian, while the third ring is of projective dimension one. However, we have the following simple observation.
\begin{lemma} \label{lemma:art}
Let $I$ be a graded ideal in $S$ and suppose that there is an element $[f] \in R_i$ such that 
$[f]R_d = R_{d+i}$. Then $R$ is either artinian or of projective dimension zero. 

\end{lemma}
\begin{proof}
The ring $S/(I+(f))$ is artinian, hence $S/I$ is of at most projective dimension zero.
\end{proof}

\begin{lemma} \label{lemma:projzero}
Suppose that $(f_1, \ldots, f_n) = I$ is generated by $n$ elements in $\Bbbk[x_0, \ldots,x_n]$ and that there is an element 
$[f] \in R_i$ such that 
$[f]R_d = R_{d+i}$. Then $R$ is of projective dimension zero.
\end{lemma}
\begin{proof}
The ring $S/(I+(f))$ is artinian, hence $(f_1, \ldots, f_n,f)$ forms a regular sequence. But also
$(f_1, \ldots, f_n)$ forms a regular sequence, so $S/I$ is of projective dimension zero.
\end{proof}

Even if we know that $R$ is of projective dimension zero, it is also hard to tell whether or not 
the maximal ideal is associated. The following example shows that although $\dim_{\Bbbk}(R_{d}) = \dim_{\Bbbk}(R_{d+1})$ and 
there is an element $l$ such that $[l]R_d = R_{d+1}$, it does not hold that $d \geq \nz(R)$.  

\begin{example} \label{example:localnonzerodivisor}
Let $I = (x^2-xz, xy-z^2, y^2-z^2)$. Then $\Hs(R,t) = 1 + 3t + 3t^2 + t^3 + t^4 + \cdots$ and $\nz(R) = 3$. We have  
$I = (x-y,x-z) \cap (z^2, y^2, xy, x^2-xz)$, $V(I) = V((x-y,x-z)) = (1:1:1)$ and
$\sqrt{(z^2, y^2, xy, x^2-xz)} = (x,y,z)$. We can choose $\{[x],[y],[z]\}$ och $\{[xz],[yz],[z^2]\}$ as $\Bbbk$-bases in degree one and two respectively and thus, 
the map from $R_1$ to $R_2$ induced by multiplication by $[z]$ is injective.
% > ideal I = (x^2-xz,xy-z^2,y^2-z^2)
% > LIB "primdec.lib";
% > primdecGTZ(I);
\end{example}

Fortunately, as the next theorem shows, if we are only interested in computing the variety, it is enough to find an $l$ such that $[l]R_d = R_{d+1}$. 

 \begin{theorem} \label{thm:gen}
 
 Let $I$ be any homogeneous ideal and let $R = S/I$. Suppose that there exists an element $[l]$ such that 
 $[l]R_d = R_{d+1}$. Let $[f_1], \ldots, [f_{t}]$ be a $\Bbbk$-basis for $R_{d+1}$. 
 Let $[e_1], \ldots, [e_t]$ be such that  $[e_i] [l] = [f_i]$. 
 Let $A_0, \ldots, A_{n}$ be such that $A_i$ corresponds to multiplication with $x_i$ with respect to the bases $[e_1], \ldots, [e_t]$ and $[f_1], \ldots, [f_t]$. 
 
 Suppose that $p \in V(I)$. Then $e(p)^t = (e_1(p), \ldots, e_t(p))^t$ is a common eigenvector to the $A_i$'s. Let $\lambda_i$ be the eigenvalue of $A_i$ corresponding to $e(p)^t$. 
 Then $p = (\lambda_0:\lambda_1: \cdots : \lambda_n)$.
  \end{theorem}
 
 \begin{proof}
 
 By the definition of the matrix $A_j$ we have
 $$[x_j][e_k] = a_{k1}^{(j)} [l e_1]+ \cdots + a_{kt}^{(j)} [l  e_t].$$

Thus, we get
 $$x_j(p) e_k(p) = a_{k1}^{(j)} l(p) \cdot e_1 (p)+ \cdots + a_{kt}^{(j)}l(p) \cdot e_t(p),$$
 or put in matrix form
 \begin{equation}
 x_j(p)e(p)^t = A_j l(p) e(p)^t.
 \end{equation}
Now $e(p)^t$ can not be the zero vector, since otherwise we would have $p \in V(I + (e_1) + \cdots + (e_t))$, which is a contradiction since 
$S/(I + (e_1) + \cdots + (e_t))$ is artinian. With the same argument, $l(p)$ must be non-zero.

 %Also $l(p)$ must be non-zero, since outerwise we would have $p \in V(I + (l))$, which is a contradiction since $S/(I + (l))$ is artinian. 
 %is zero, then $y_j(p) = 0$, for $j=0, \ldots, h$. By Lemma \ref{nonzero} $l(p)$ is non-zero.  
 
 Hence $\frac{e(p)^t}{l(p)}$  is an eigenvector of $A_j$ with the
 eigenvalue $\frac{x_j(p)}{l(p)}$. The theorem follows since 
 \begin{gather*}
 (\lambda_{0}: \lambda_{1} : \cdots : \lambda_{h}) = (x_0(p)/l(p): x_1(p)/l(p) : \cdots : x_n(p)/l(p)) \\=
 (x_0(p): x_1(p): \cdots : x_n(p)).
 \end{gather*}
 \end{proof}

\begin{example}
Let $I = (y^2,z^2,xz,xy)$. Then $\dim_{\Bbbk}(R_1) = 3$ and $\dim_{\Bbbk}(R_2) = 2.$
%$\Hs(R,t) = 1 + 3t + 2t^2 + t^3 + t^4 + \cdots$.
We can choose $[x],[y],[z]$ and $[x^2], [yz]$ as $\Bbbk$-bases in degrees one and two respectively. It is clear that
$[x+z]R_2 = R_3$ and thus, by Lemma \ref{lemma:art}, we know that $I$ is of at most projective dimension zero. We have 
$[x+z] [x]= [x^2]$ and $[x+z] [y] = [yz]$, so with respect to the bases $[x], [y]$ and $[x+z] [x],  [x+z] [y]$, we get
$$
A_x =
\begin{pmatrix}
1 & 0 \\
0 & 0
\end{pmatrix},
A_y =
\begin{pmatrix}
0 & 0 \\
0 & 0 
\end{pmatrix} \text{ and }
A_z =
\begin{pmatrix}
0 & 0 \\
0 & 1 
\end{pmatrix}.
$$
There are two common eigenvectors for these matrices --- $(1,0)$ and $(0,1)$. The associated eigenvalues are
$1,0,0$ and $0,0,1$ respectively. By Theorem \ref{thm:gen}, we know that $V(I) \subseteq \{(1:0:0), (0:0:1)\}$. 
We have $y^2((1:0:0)) = z^2((1:0:0)) = xz((1:0:0)) = xy((1:0:0)) = 0$, but $z^2((0:0:1))  \neq 0$, so $V(I) = \{ (1:0:0)\}$. Thus, the second point 
was "false". 
%The primary decomposition of the ideal is $(y,z) \cap (x,y^2,z^2)$. 

\end{example}

Theorem \ref{thm:gen} could also be used to compute the variety in Example \ref{example:localnonzerodivisor}. We leave this computation as an exercise to the reader. 

So suppose that we want to compute the variety of an ideal $I$ which we suspect is of projective dimension zero. We propose the following procedure for a field 
$\Bbbk$ with enough elements.
\begin{itemize}
\item[\textbf{K1}]Compute the Gr\"obner basis elements of degree $1, 2$ and so on until we reach a degree $d$ such that 
$|\ini(I)_d^c| \geq \ini(I)_{d+1}^c|$ (this is the same as $\dim_{\Bbbk}(R_d) \geq \dim_{\Bbbk}(R_{d+1})=t$).
\item[\textbf{K2}] Choose a linear form $[l]$ at random and check if $R_d [l] = R_{d+1}$. If it was not, choose another $l$. If we did not find such an 
element even after many tries, go back to stage K1 and compute more Gr\"obner basis elements. 
\item[\textbf{K3}] Choose a basis $\{[f_1], \ldots, [f_t]\} = [\ini(I)_{d+1}^c]$ for $R_{d+1}$ and let $e_1, \ldots, e_t$ be such that $[e_i] [l] = f_i$.

\item[\textbf{K4}] Compute the multiplication matrices with respect to  $\{[e_t], \ldots, [e_t]\}$ and $\{[f_1], \ldots, [f_t]\}$.

\item[\textbf{K5}] Determine a set of common eigenvectors for these matrices, either by using symbolic or numerical methods and use Theorem \ref{thm:gen} to determine the variety of $I$.
\end{itemize}

We refer the reader to the book \cite{Stetterbook} and the citations therein for techniques to compute common eigenvectors using numerical methods.

\subsection{Computing $(\tilde{R},A,l)$ from the points} \label{sec:pointstovspace}

Given a set of projective points $P = \{p_1, \ldots, p_m\}$ one can form the vanishing ideal $I(P)$, which consists of all polynomials vanishing on all of the points in $P$. 
The Hilbert series of $R/I(P)$ is well studied but not completely understood, cf. \cite{geramita}. 
The most common way of computing Hilbert series of an ideal defined by projective points has been studied by means of the projective Buchberger-M\"oller algorithm 
\cite{Abbottetal, MMM}. This algorithm  computes a Gr\"obner basis of a vanishing ideal by 
computing a $\Bbbk$-basis for the $\Bbbk$-spaces $R_0, R_1, \ldots, R_d$ until degree $\max(m,\nz(R))$ and reducing potential 
Gr\"obner basis elements with respect to this basis.  We will present a reduced version of the projective Buchberger-M\"oller algorithm which instead of computing
the Gr\"obner basis of $I(P)$ computes the triplet $(\tilde{R},A,l)$ and 
we will show that the behavior of our method is asymptotically better than the classical Buchberger-M\"oller-algorithm.

Recall that we suppose that the representation of each projective point is fixed so that we can define evaluation of projective points in a unique way.

A nice way to compute normal forms with respect to vanishing ideals of projective points is by evaluation: Given a form $f$ of degree $d$ and a vector
space basis $e_1,\ldots, e_m$ of $R_d$, we obtain the normal form $\alpha_1 e_1 + \cdots + \alpha_m e_m$, where the $\alpha_i$'s are chosen to satisfy
$f(p_i) = \alpha_1 e_1(p_i) + \cdots + \alpha_m e_m(p_i)$ for $i = 1, \ldots, m$. The normal form does not depend on the choice of representation of the points.
Computing normal forms by means of evaluation is the key engine behind the graded Buchberger-M\"oller algorithm and the variation of the method given below.

When studying ideals of projective points, one can always assume that $n+1 \leq m$. Indeed, we have the following lemma, which is a graded version of Lemma 5.2 in \cite{lundqvist2}.

\begin{lemma} \label{lemma:gradedprojver}

Let $E = \{ x_{i_0}, \ldots, x_{i_{\overline{n}}} \}$ be  any subset of the variables such that
 $E(P)$ and $\{x_0, \ldots, x_n\}(P)$ has same rank. Let  
$\pi$ be the natural projection from $P^n(\Bbbk)$ to $P^{\overline{n}}(\Bbbk)$ defined by 
$\pi((a_0: \cdots: a_n)) = (a_{i_0}: \cdots:a_{i_{\overline{n}}})$. Then  $q_1, \ldots, q_m$ are distinct where  $q_i=\pi(p_{i})$.  Moreover, with $Q = \{q_1, \ldots, q_m\}$ and with 
$\overline{R} = \Bbbk[x_{i_0}, \ldots, x_{i_{\overline{n}}}]/I(Q)$,  the graded algebras 
 $\overline{R}$ and $R$ are isomorphic. 
\end{lemma}
\begin{proof}
Suppose that $x_i \notin E$. Then $x_i(P) = \alpha_1 x_{i_1}(P) + \cdots + \alpha_{\overline{n}} x_{i_{\overline{n}}}(P)$. Hence
$x_i -  \alpha_1 x_{i_1} + \cdots + \alpha_{\overline{n}} x_{i_{\overline{n}}} \in I$. Since $R_{\geq 1}$ is generated in degree one, it is clear that the elements in $E$ generates $R_{\geq 1}$. 
Since the evaluation on $q_i$'s and the $p_i$'s agrees on the elements in $\Bbbk[x_{i_0}, \ldots, x_{i_{\overline{n}}}]$, it follows that 
$\overline{R}$ and $R$ are isomorphic as graded algebras. It is the clear that
$q_1, \ldots, q_m$ are distinct.
\end{proof}

\begin{remark}
In a more subtle way, Lemma \ref{lemma:gradedprojver} actually follows directly from the projective Buchberger-M\"oller algorithm.
\end{remark}

We now give a variant of the projective Buchberger-M\"oller algorithm for building the triplet $(\tilde{R},A,l)$ from the points. As for the Buchberger-M\"oller algorithm, this algorithm is based on the evaluation method to compute normal forms. But it differs from the Buchberger-M\"oller algorithm in the sense that it is focused on giving the multiplication tables with respect to the variables rather than giving a Gr\"obner basis for the ideal.

\vspace{0.5cm}

\noindent

\begin{itemize}

%\item[\textbf{L1}]
%If $n+1>m$, row reduce the matrix $\{x_0, \ldots, x_n\}(P)$ to find
%a subset $E = \{x_0, \ldots, x_{\overline{n}}\}$ of the variables and linear %relations modulo $I(P)$ in the sense that
%$x_j = \sum_k c_{jk}i x_{i_k} \pmod I(P)$ for $j = 0, \ldots, n$.
%Let $Q_1, \ldots, Q_m$ be the projection of the $p_i$'s with respect to 
%$E$.

\item[\textbf{L1}]
Compute a non-zero divisor $l$ of degree one by using the method in Proposition \ref{prop:nonzerodivisordegone}.

\item[\textbf{L2}] 
Initiate the lists $B_0 = L_0 = [1]$ and $\initials = []$. Let $d =0$.

\item[\textbf{L3}] 
If $\rank (B_d(P)) = |P|$, then $\nz(R) = d$. Return $B_0, \ldots, B_d$ and $l$.  
Otherwise, increase $d$ by one, let $B_d = [ ]$ and let $L_d$ be the list of all monomials of degree $d$ which are not multiples of an element of $\initials$.

\item[\textbf{L4}] 
If $L_d$ is empty, go to step L3; otherwise choose the monomial 
$t = min_{\prec}(L_d)$ with respect to a fixed monomial order and remove it from $L_d$.

\item[\textbf{L5}] If $t(P)$ can be written as a linear combination of the rows in $B_d(P)$, then
add $t$ to the set Initials and continue with step L4. Else, append $t$ to $B_d$ and continue with step L4. 
\end{itemize}
The correctness of the method is a direct consequence of the projective BM-algorithm, since the sets $B_0, \ldots, B_d$ are computed in the same way using the two methods. 
Thus $\initials$ will generate $\ini(I)_{\prec}$, while $B$ will be the complement of $\ini(I)_{\prec}$. By using another selection method in step L4, it is possible to obtain a basis which is not necessarily the complement of an initial ideal (It is an easy exercise to check that the termination of the algorithm does not depend on the selection method). 

We implicitly assume that we have used Lemma \ref{lemma:gradedprojver} so that $n \leq m$. This preprocessing can be done using $O(nm^2)$ arithmetic operations, since we test for linear
dependence $n$ times.
It is straightforward to lift the result in \cite{lundqvist1} and show that the complexity of the algorithm is dominated by the arithmetic operations and not the monomial manipulations. 
The number of arithmetic operations for the step L1 is bounded by $O(m^3)$ by an elementary analysis of the method in Proposition \ref{prop:nonzerodivisordegone}. For each degree $d$ during the algorithm, we need to check linear dependence at most $\min(m,n) m$ times. Thus, the arithmetic complexity of the method is a cost of at most $O(\post(R) \min(m,n) m^3)$. The original analysis \cite{MMM} of the Buchberger-M\"oller algorithm reports the complexity $O(nm^4)$. (Although it is possible to show that the performance is $O(\min(m,n)nm^3)$.)

%The multiplication table algorithm is however much cleaner than the Buchberger-M\"oller approach.

%
%The form in which we present the algorithm in is a little bit improcative, since it totally ignores the computation of ideal generators (or a Gr\"obner  basis). It is however possible to get generators for the ideal by  

\begin{example} \label{ex:multalg}
In Example 3.5 in \cite{Abbottetal}, the point set 
$P =\{(0:2:5), (0:1:2)$, $(1:3:1),$ $(4:3:4)$, $(2:5:4),
(1:4:4)\}$ is considered. The Gröbner basis with respect to DegRevLex and 
$x \prec y \prec z$ is generated in degrees three and four and $\Hs(R,t)= 1+3t+6t^2+6t^3+\cdots$. 

With our approach, we would first fix the coordinates $P =\{(0,2,5), (0,1,2)$, $(1,3,1),$ $(4,3,4)$, $(2,5,4),
(1,4,4)\}$ and then compute $L_0= \{1\}, B_0 = \{1\}$ and $L_1 = \{x,y,z\}$. Since $\{x,y,z\}(P)$ has full rang, we will have $B_1 = \{x,y,z\}$. Thus, Initials is empty and 
$L_2 = \{ z^2, yz, xz, y^2, xy, x^2 \}$. It turns out that also $\{z^2, yz, xz, y^2, xy, x^2 \}(P)$ has full rang, so
$B_2 =\{ z^2, yz, xz, y^2, xy, x^2\}$. Since $|P| = 6$, the algorithm stops and we know that $\Hs(R,t) = 1 + 3t + 6t^2 + 6t^3 + \cdots$. 
It is immediate that 
$y(p_i) \neq 0$ for $i = 1, \ldots, 6$, so we can use $B_3 =\{yz^2, y^2z, xyz, y^3, xy^2, x^2y \}$ as a $\Bbbk$-basis in degree three, and in general
$$B_d =\{y^{d-2}z^2, y^{d-1}z, xy^{d-2}z, y^{d}, xy^{d-1}, x^2y^{d-2} \}.$$ 

Say that we want to compute the normal form of $x^6 + z^6$.
If we do it by evaluation, we solve the linear equations
$$(x^6+ z^6)(p_i) = (\alpha_1 y^4z^2 +  \alpha_2 y^{5}z +  \alpha_3 xy^{4} + \alpha_4 y^{6} + \alpha_5 xy^{5}+ \alpha_6 x^2y^{4})(p_i)$$ for $i = 1, \ldots, 6$ which is equivalent to 
perform Gaussian elimination on a $(6 \times 6)$-matrix. As result, we get 
$$\nf(x^6 + z^6, B) = \frac{2083926583 }{23522400}y^6-\frac{11603225231}{470448000} x
    y^5-\frac{8111541583 
    }{26136000}y^5z$$ $$-\frac{327280970021}{940896000} x^2
    y^4 +\frac{17527852333 
   }{117612000} y^4z^2+\frac{127511218609 }{313632000}x y^4z.$$

We could also use the multiplication matrices. Notice that $A_x, A_y$ and $A_z$ share the six
linear independent eigenvectors $$(z^2(p_i),  yz(p_i), xz(p_i), y^2(p_i), xy(p_i), x^2(p_i))^t.$$ Thus, if we let
$T = B_2(P)^t$, we have $A_x = T D_x T^{-1}$, $A_y = T D_y T^{-1}$ and $A_z = T D_z T^{-1}$, where 
$D_x = \diag(0,0,1/3,4/3,2/5,1/4)$, $D_y = \diag(1,1,1,1,1,1)$ and $D_z = \diag(5/2,2,1/3,4/3,4/5,1)$.
  
So to compute the normal form of $x^6 + z^6$, we can start by computing
$\nf(x^2,B) = x^2 = (0,0,0,0,0,1) \cdot  (z^2, yz, xz, y^2, xy, x^2) $ and $\nf(z^2,B) = z_2 = (1,0,0,0,0,0) \cdot (z^2, yz, xz, y^2, xy, x^2)$. We get
$\nf(x^6+z^6,B) = \nf(x^6,B) + \nf(z^6,B)$, where 
$$\nf(x^6,B) = (0,0,0,0,0,1) A_x^4 (y^4z^2, y^5z, xy^4z, y^6, xy^5, x^2y^4)^t$$ 
$$= (0,0,0,0,0,1) TD_x^4T^{-1} (y^4z^2, y^5z, xy^4z, y^6, xy^5, x^2y^4)^t$$ 
and
$$\nf(z^6,B) = (1,0,0,0,0,0) A_z^4 (y^4z^2, y^5z, xy^4z, y^6, xy^5, x^2y^4)^t$$ 
$$= (1,0,0,0,0,0) TD_z^4T^{-1} (y^4z^2, y^5z, xy^4z, y^6, xy^5, x^2y^4)^t.$$ 
Of course, we get the same result as above.

%For instance $$A_x = \left(
%\begin{array}{llllll}
% \frac{25}{4} & 4 & \frac{1}{9} & \frac{16}{9} & \frac{16}{25} & 1 \\
%\frac{5}{2} & 2 & \frac{1}{3} & \frac{4}{3} & \frac{4}{5} & 1 \\
% 0 & 0 & \frac{1}{9} & \frac{16}{9} & \frac{8}{25} & \frac{1}{4} \\
% 1 & 1 & 1 & 1 & 1 & 1 \\
% 0 & 0 & \frac{1}{3} & \frac{4}{3} & \frac{2}{5} & \frac{1}{4} \\
% 0 & 0 & \frac{1}{9} & \frac{16}{9} & \frac{4}{25} & \frac{1}{16}
%\end{array}
%\right),$$
%so 
%$$\nf(xz^2,B_3) =   \frac{25}{4} yz^2 + 4 y^2z+ \frac{1}{9} xyz + \frac{16}{9} y^3 + \frac{16}{25} xy^2 + %x^2y.$$

% If we whish to compute the normal form of $x^5y^7z^8$ with respect to 
% $B_{23} = 

% It should be remarked that Example 3.5 contains a miscalc

% we consider the matrix 
% $L_3(P) = 
% 
% \left(
% \begin{array}{llllll}
%  125 & 8 & 1 & 64 & 64 & 64 \\
%  50 & 4 & 3 & 48 & 80 & 64 \\
%  0 & 0 & 1 & 64 & 32 & 16 \\
%  20 & 2 & 9 & 36 & 100 & 64 \\
%  0 & 0 & 3 & 48 & 40 & 16 \\
%  0 & 0 & 1 & 64 & 16 & 4
% \end{array}
% \right)

\end{example}

\subsubsection{Computing separators} \label{sec:computingseparators}

A family of separators with respect to a set of affine points $P = \{p_1, \ldots, p_m\}$, 
is a set $\{f_1,\ldots, f_m\}$ of polynomial functions such that $f_i(p_i) = 1$
and $f_i(p_j) = 0$ if $i \neq j$. It is easy to see that the
separators forms a $\Bbbk$-basis for $\Bbbk[x_1, \ldots, x_n]/I(P)$. 

When the points are projective, we say that  $\{f_1,\ldots, f_m\}$ is a set of separators 
if $f_i(p_i) \neq 0$ and $f_i(p_j) = 0$ when $i \neq j$. When all separators are 
of the same degree $d$, they constitute a $\Bbbk$-basis for $R_d$. If $\Bbbk$ contains at least $m$ elements so that there exists a non-zero divisor $l$, 
we can construct a separator-$\Bbbk$-basis for $R_{d+i}$ as $f_1l^i, \ldots, f_m l^i$. 

In \cite{lundqvist2}, two methods for computing separators with respect to a collection of affine points are discussed. It is possible to lift this method to the
projective setting.
Both methods perform the same number of arithmetic operations.
% (at most $nm + m\min(m,nr)$ --- the constant $r$ will be defined ).
We will illustrate one of the methods by an example. 
In $\mathbb{Z}^5$, consider
  $p_1 = (1,2,0,1,1),$
  $p_2 = (1,0,1,1,2),$
  $p_3 = (1,2,0,3,3),$
  $p_4 = (0,0,2,0,4),$ 
  $p_5 = (0,0,2,1,5)$ and
  $p_6 = (2,1,3,1,6)$. We associate the following table to this point set  \begin{displaymath}
\left( \begin{array}{cccccc|lc}
  1 & 1 & 1 & 0 & 0 & 2 & \Sigma_1 = & \{\{1,2,3 \}, \{4,5 \}, \{6 \} \}  \\
      2 & 0 & 2 & 0 & 0 & 1 & \Sigma_2 = & \{\{1,3\}, \{2\}, \{4,5\}, \{6\} \} \\
      0 & 1 & 0 & 2 & 2 & 3 & \Sigma_3 = &\{\{1,3\}, \{2\}, \{4,5\}, \{6\} \} \\
      1 & 1 & 3 & 0 & 1 & 1 & \Sigma_4 = &\{\{1 \}, \{3 \}, \{2\}, \{4\}, \{5\}, \{6\} \}\\
      1 & 2 & 3 & 4 & 5 & 6 & \Sigma_5 = &\{\{1 \}, \{3 \}, \{2\}, \{4\}, \{5\}, \{6\} \}
\end{array} \right)
\end{displaymath}
The sets on the right hand side are also described by an example: The set $\{1,3\}$ on the second row shows that $p_1$ and $p_3$ agree on the first two coordinates. When computing such a table from a point set, one obtains a matrix 
$c_{ij}$, where $c_{ij}$ is the first position where $p_i$ and $p_j$ differ. 

This matrix is used to compute the separators and it is clear that 
$$Q_i = \prod_{i \neq j} \frac{x_{c_{ij}}-p_{jc_{ij}}}{p_{ic_{ij}}-p_{jc_{ij}}}$$
satisfies $Q_i(p_j)= 0$ if $i \neq j$ and $Q_i(p_i) = 1$. 

It is showed in \cite{lundqvist2} that at most $nm + m^2$ arithmetic comparisons are used to compute the matrix $c_{ij}$. (In fact a slightly improved upper bound is given.)

% where $r$ is  branches going 

%On the first row, we compare $v_{11}$ and $v_{12}$, 
%$v_{11}$ and $v_{13}$, $v_{11}$ and $v_{14}$, 
%$v_{11}$ and $v_{15}$ and finally $v_{11}$ and $v_{16}$. From this we
%know that the equivalance class on the first level containing $v_1$ is 
%$[v_1, v_2, v_3]$. We continue by
%comparing $v_4$ and $v_5$ and $v_4$ and $v_6$. From this we conlude that
%$[v_4, v_5]$ an d $[v_6]$ are the remaining equivalance classes 
%on the first level. 
%On the second level, we repeat the procedure given above for 
%each equivalance class on the first level, e.g., we compare
%$v_{12}$ and $v_{22}$, $v_{12}$ and $v_{32}$ and conclude that
%$[v_1,v_3]$ and $[v_2]$ are equivalance classes with respect to the
%first two coordinates. 
%Morever, the algorithm computes a matrix $(c_{ij})$, where 
%$c_{ij}$ denotes the first level for which $v_i$ and $v_j$ were put
%in different equivalance classes. For instance, we have $c_{12} = 4$ since $v_{1i} = %v_{2i}$ for $i = 1,2,3$, but $v_{1i} \neq v_{2i}$. 

%It was shown in \cite{lundqvist2} that this method uses 
%at most $nm + m\min(m,nr)$ arithmetic comparisons, 
%where $r$ is the maximal number
%of equivalance classes constructed on the $i$'th level from one
%equivalance class on the $i-1$'th level. In the example above, $r = 3$. 

We will now show how to make use of the matrix $c_{ij}$ to compute projective separators.
If we let $S_{ij}(p_i) \neq 0$ and $S_{ij}(p_j) = 0$, then
$Q_1, \ldots, Q_m$ is a set of projective separators for $p_1, \ldots, p_m$, 
where
\begin{equation} \label{expr:sep} 
Q_i = \prod_{j\ne i}  S_{ij}.
\end{equation}

Suppose that 
each point $p_i$ is normalized in the sense that the first non-zero position equals one. It is then clear that we can use the affine method to compute the matrix $(c_{ij})$ with respect to the points.

We will now give an explicit formula for each $S_{ij}$. 
To simplify notation, let $h = c_{ij}$.

\begin{itemize}
\item If  $p_{ih} = 0$, then $p_{jh} \neq 0$. 
Let $h'$ be the least position such that
 $p_{i h'} = 1$ and let 
$S_{ij} = p_{j h} x_{h'} - p_{j h'}x_h$. 
\item 
Else, if $p_{ih} \neq 0$ but $p_{jh} = 0$, then let
$S_{ij} = x_h$. 
\item
Finally, suppose that
$p_{ih} \neq 0$ and $p_{jh} \neq 0$. Since $p_i$ and $p_j$ agrees on all coordinates less than $h$ and $p_{ih} \neq p_{jh}$, there is a 
$h' \leq h$ such that $p_{i h'} = p_{j h'} = 1$. Thus, let
$S_{ij} = p_{jh} x_{h'} - p_{j h'} x_h = p_{jh} x_{h'} - x_h$.
\end{itemize}
Notice that we can choose the index $h'$ occurring in the two situations as the first entry where $p_i$ equals one. It is clear that we can determine the first non-zero index of each point using at most $nm$ arithmetic comparisons. We have proved the following theorem.

\begin{theorem}  \label{thm:separators}
Let $P = \{p_1, \ldots, p_m\}$ be a set of distinct projective points. Then we can compute a set of separators of degree $m-1$ with respect to $P$ using at most
$nm + m^2$ arithmetic operations. 
\end{theorem}

\begin{example}

Let $p_1 = (1:2:0:1:1:0:3:5),$
  $p_2 = (1:0:1:1:2:0:3:5),$
  $p_3 = (1:2:0:3:3:1:2:0)$ and let
  $p_4 = (0:1:1:0:2:0:1:0).$
We will show how to compute $Q_1$. We have $c_{12} = 2$ and 
$p_{12} = 2$ and $p_{22} = 0$. Thus, $S_{12} = x_2$. 
We have $c_{13} = 4$ and $p_{14} = 1$ and $p_{34} = 3$. Since
$p_{11} = p_{31} = 1$, we let 
$S_{13} = p_{34} x_1 - x_4 = 3x_1-x_4.$
We have $c_{14} = 1$ and $p_{41} = 0$, so $S_{14} = x_1$. Hence
$Q_1 = x_2(3x_1-x_4)x_1.$

\end{example}

\section{Discussion}
%\subsection{Rings of arbitrary projective dimension}

In a forthcoming paper we will generalize parts of the results to rings of arbitrary projective dimension.

% \section{Acknowledgments}
% The author would like to thank Clas Löfwall for posing the question "Let $A_x^i$ be the map from "
% 


\begin{thebibliography}{99}
\bibitem{Abbottetal}
J. Abbott, A. Bigatti, M. Kreuzer and L. Robbiano, Computing ideals of points, J. Symb. Comput. 30, 2000, pp. 341-356.

\bibitem{Atiyah}
M. F. Atiyah, I. G. Macdonald. Introduction to commutative algebra. 
Addison-Wesley Publishing Co., Reading, Mass., 1969.

\bibitem{BM}
B. Buchberger and M. M\"{o}ller, The construction of multivariate polynomials with preassigned zeroes, Computer
Algebra: EUROCAM .82 (J. Calmet, ed.), Lecture Notes in Computer Science 144, 1982, pp. 24-31.

\bibitem{matrix}
D. Coppersmith, S. Winograd, Matrix multiplication via arithmetic progressions, J. Symb. Comput. 9, 1990, pp. 251-280.

\bibitem{corless} R.M. Corless, Editor's Corner: Gr\"{o}bner bases and
  Matrix Eigenproblems. SIGSAM Bull. 30, nr 4, 1996, pp. 26-32. 

\bibitem{Corlessetal}
R.M. Corless, P.M. Gianni, B.M. Trager, A reordered Schur factorization method for zero-dimensional polynomial systems with multiple roots, ISSAC '97: Proceedings of the 1997 international symposium on Symbolic and algebraic computation, 1997.

\bibitem{FGLM}
J.C. Faug\`{e}re, P. Gianni, D. Lazard and T. Mora, Efficient Computation of 
Zero-Dimensional Gr\"{o}bner Basis by Change of Ordering. J. Symb. Comput. 16, 1993, pp. 329-344.
\bibitem{geramita} A.V. Geramita, P. Maroscia, L. Roberts, Hilbert Function of a Reduced k-algebra, J. London Math. Soc., (2), 28, 1983, pp. 443-452.
\bibitem{gotzmann} G. Gotzmann, Eine Bedingung f\"ur die Flachheit und das Hilbertpolynom eines graduierten Ringes, Math. Z. 158, 1978, pp. 61-70.

\bibitem{Grobner} W. Gr\"obner, Algebraische Geometrie, Vol 2, Bibliographisches Institut Mannheim (1967-1968).
 
%\bibitem{LaubStigler04}
%R. Laubenbacher, B. Stigler, A computational algebra approach to the reverse-engineering of gene %regulatory networks. J. Theor. Biol. 229, 2004, pp. 523-537.

\bibitem{lundqvist1}
S. Lundqvist, Complexity of comparing monomials and two improvements of the Buchberger-M\"oller algorithm,
Proceedings of MMISC 2008 (J. Calmet,
               W. Geiselmann and
               J. M{\"u}ller-Quade ed.), Lecture Notes in Computer Science 5393, 2008, pp. 105-125.
\bibitem{lundqvist2}
S. Lundqvist, Vector space bases associated to vanishing ideals of points, 2008, arXiv:0808.3591. Submitted to Journal of Pure and Applied Algebra.

\bibitem{lundqvist3}
S. Lundqvist, Non-vansihing forms in projective space, in preparation.

\bibitem{MMM}
M. G. Marinari, H.M. M\"{o}ller, and T. Mora, Gr\"{o}bner bases of ideals defined by functionals with an application to ideals of projective points. Applicable Algebra in Engineering, Communication and Computing 4, 1993, pp. 103-145.

\bibitem{MMMtrans}
M. G. Marinari, H.M. M\"{o}ller, and T. Mora, On multiplicities in polynomial system solving. Trans. Amer. Math. Soc, vol 348,1996, pp. 3283--3321.

\bibitem{Mollerstetter}
H.M. M\"{o}ller and H. Stetter, Multivariate polynomial equations with multiple zeros solved by matrix eigenproblems, Numer. Math. 70, 1995, pp. 311-329.


\bibitem{Mollertenberg}
H.M. Möller, R. Tenberg. Multivariate polynomial system solving using intersection of eigenspaces. J. Symb. Comput. 30, 2001 pp. 1--19.

\bibitem{robbiano} L. Robbiano, Introduction to the theory of Hilbert functions, Queen's Papers in Pure and Applied Math., vol. 85 (1991) B1--B26.

\bibitem{Stetterbook}  H. Stetter, Numerical Polynomial Algebra, xv+472 pp., Philadelphia: SIAM, 2004. 

\end{thebibliography}
\end{document}